\documentclass[11pt,a4paper]{article}
\usepackage{graphicx,import}
\usepackage{color}
\usepackage{xcolor}

\usepackage[utf8]{inputenc}
\usepackage[normalem]{ulem}
\usepackage{amsmath,amsthm,amsfonts} 
\usepackage{hyperref}

\newcommand\R{\mathbb{R}}
\newcommand\Z{\mathbb{Z}}

\newcommand\eps{\varepsilon}
\newcommand\Div{\mathop{\mathrm{div}}}

\newtheorem{theorem}{Theorem}[section]
\newtheorem{lemma}[theorem]{Lemma}
\newtheorem{definition}[theorem]{Definition}
\newtheorem{remark}[theorem]{Remark}

\newtheorem{proposition}[theorem]{Proposition}
\numberwithin{equation}{section}

 \usepackage{fancyhdr}
\usepackage[english]{babel}

\usepackage[OT1]{fontenc}
\usepackage[T1]{fontenc}
\usepackage{amsfonts}
\usepackage{amsmath}
\usepackage{amssymb}
\usepackage{amsthm}
\usepackage{float}
\usepackage{textcomp}
\usepackage{xcolor}
\usepackage{graphicx}
\usepackage{amssymb}
\usepackage{fancybox}
\usepackage{mathrsfs}
\usepackage{mathtools}
\usepackage{tikz}
\usepackage{centernot}
\usepackage{listings}
\usepackage{faktor}
\usepackage{dsfont}
\usepackage{hyperref}
\usepackage{units}
\usepackage{color}
\usepackage{csquotes}
\usepackage{verbatim} 
\usepackage{ stmaryrd }
\usepackage{amsthm}
\usepackage{url}
\usepackage{float}
\usepackage{textcomp}
\usepackage{algpseudocode}
\usepackage{framed}
\usepackage{halloweenmath}

\usepackage{MnSymbol}

\usepackage{authblk}

\usepackage[italiano, boxed ]{algorithm2e}

\graphicspath{ {./images/} }
\usepackage{subfiles}

\newcommand{\C}[1]{\mathcal{#1}}

\newcommand{\B}[1]{\mathbb{#1}}

\newcommand{\BR}{\mathbb{R}}

\newcommand{\mres}{\mathbin{\vrule height 1.6ex depth 0pt width
0.13ex\vrule height 0.13ex depth 0pt width 1.3ex}}

\newcommand{\wsconv}{\overset{\ast}{\rightharpoonup}}

\newcommand{\ha}{\mathcal{H}}

\newcommand{\mes}{\C M (\Omega)}

\DeclareMathOperator*{\Divergenza}{div}

\DeclarePairedDelimiter{\braket}{ ( } { )}

\title{Homogenization of line tension energies}
\author{M. Fortuna$^1$,
  A. Garroni$^2$\\
{\small Dipartimento di Matematica ``Guido Castelnuovo'' \\
	Sapienza Università di Roma
P.le A. Moro 5 00185 Roma Italy}}
\date{ }

\begin{document}
	\pagestyle{plain}

\maketitle
\begin{center}
\begin{minipage}[c]{0.8\textwidth}
Abstract:	We prove an homogenization result, in terms of $\Gamma$-convergence, for energies concentrated on rectifiable lines in $\R^3$ without boundary. 
The main application of our result is in the context of dislocation lines in dimension $3$.  The result presented here shows that the line tension energy  of  unions of single line defects converge to the energy associated to macroscopic densities of dislocations carrying  plastic deformation. As a byproduct of our construction for the upper bound for the $\Gamma$-Limit, we obtain an alternative proof of the density of rectifiable $1$-currents without boundary in the space of divergence free fields.
\end{minipage}
\end{center}
\section*{Keywords}
Gamma convergence, dislocations, divergence free fields, homogenization.

\section{Introduction}
\footnotetext[1]{martino.fortuna@uniroma1.it}
\footnotetext[2]{garroni@mat.uniroma1.it}

In this paper we prove an homogenization result for energies of the form
\begin{equation}\label{definition form of the energy}
\int_\gamma \psi(b,t) d \ha^1,
\end{equation}
where $\mu=b\otimes t\ha^1 \mres \gamma$ is a divergence free matrix valued measure, $\gamma\subset \BR^3$ is a $1$-rectifiable set, and $t$ its tangent. The vector $b$ is a  multiplicity which belongs to a discrete lattice $\C B$ in $\R^N$, with $N\geq 2$, and will also be called the Burgers vector of $\mu$.  Here $\C M_{df}^1(\Omega;\C B\times \B S^2)$ denotes the set of such measures, where $\Omega\subset \BR^3$ is an open bounded and regular set.

We consider the following scaled version of the energy in \eqref{definition form of the energy}
\begin{equation}\label{eq-Esigma}
E_\sigma(\mu):=\begin{cases}
	\displaystyle{	\int_{\gamma \cap \Omega} \sigma \psi\left(\dfrac{b}{\sigma},t\right)d\ha^1} \quad& \hbox{if }\ \mu=b\otimes t \ha^1\mres\gamma\in \C M_{df}^1(\Omega;\sigma\C B\times \B S^2),\\+\infty & \hbox{otherwise}.
\end{cases}
\end{equation} 
Under some mild assumptions on the density $\psi$ we study the asymptotics of $E_\sigma$ in terms of $\Gamma$-convergence with respect to the weak$*$ topology of measures.
The main result of the paper is that the limiting energy takes the form
\begin{equation}
	E_0(\mu)=\begin{cases}\displaystyle
		\int_\Omega g\left(\dfrac{d\mu(x)}{d\|\mu\|}\right)d\|\mu\|\quad & \hbox{if } \mu\in [\mes]^{N\times 3},\ \Divergenza \mu =0,\\{+\infty} & \hbox{otherwise},
	\end{cases}
\end{equation}
where $g: \BR^{N\times 3}\longrightarrow [0,+\infty)$ is a convex $1$-homogeneous function  defined in terms of  the density $\psi$ (see Theorem \ref{th-main} for the exact statement). 

Our result extends the result in  \cite{ConGarMul17}, where  the same problem is treated in  dimension $2$, i.e., $\Omega\subset \BR^2$, and from which we borrow several ideas for the proof. The main difference is that in dimension $2$ the support of the measures in $\C M_{df}^1(\Omega;\C B\times \B S^1)$ has codimension $1$, so that the energy in \eqref{definition form of the energy} reduces to a functional defined in the space of $SBV$ functions with values in $\C B$ (and the measure $\mu$ is nothing but the rotated gradient of the phase field $u$). This allow the authors to use tools from the Calculus of Variations for functionals defined on partitions and therefore in $SBV$ (see \cite{ambrosio2000fbv}).

In dimension larger than $2$ there is no phase field describing the admissible configuration, so  we use techniques of geometric measure theory and the analysis of  functionals defined on rectifiable currents (rephrased in terms of the measures in $\C M_{df}^1(\Omega;\C B\times \B S^2)$). Energies of the form \eqref{definition form of the energy} have been studied in \cite{ConGarMas15} where the authors give necessary and sufficient conditions for the lower semicontinuity of such functionals.

The major difficulty in the proof of the $\Gamma$-convergence is the construction for the upper bound. Implementing a standard density argument we first need to reduce to the case of measures absolutely continuous with respect to Lebesgue and having piecewise constant density. This step requires to prove that such measures are dense in energy for the limiting functional $E_0$.  The second key ingredient of the proof is then an ad hoc construction which allows to approximate divergence free piecewise constant fields with measures concentrated on polyhedral closed curves and optimal energy.
As a byproduct of the construction for the recovery sequence we thus obtain a different proof of the approximation of divergence free vector fields by means of measures defined on closed curves (see Theorem \ref{theo approximation straight lines}).  Stated by J. Bourgain and H.
Brezis in \cite{BouBre} in the context of solenoidal charges in the sense of Smirnov (see \cite{Smir}), this density result is  proved in \cite{GoHeSp} with respect to the strict topology of measures. \\

The main motivation for our analysis is the study of line defects in a $3$-dimensional crystal, the so called dislocations. At a mesoscopic scale (larger than the microscopic lattice spacing) they are indeed identified with line objects carrying a vector multiplicity belonging to the lattice, so that they can be represented as measures belonging to $\C M_{df}^1(\Omega;\C B\times \B S^2)$.  The divergence free constraint is reminiscent of the topological nature of these defects.
At this level one can associate to each dislocation a line tension energy, i.e., an energy with the same form of \eqref{definition form of the energy}.
Such energies can in turn be derived from more fundamental models. For example, in \cite{ConGarOr}, \cite{GarMarSca}, and \cite{ConGarMar22} the authors deduce the line tension model accounting for the elastic distortion in the material induced by the presence of dislocations. See also \cite{Po07},  \cite{ScaZep}, \cite{DeLGarPon12}, and \cite{ScaZepMu}  for a similar derivation in dimension 2 for cylindrical geometry where dislocations are viewed as point defects.
Similarly, this type of line tension model can also be derived under the assumption that the line defects are contained in a given (slip) plane as limit of nonlocal phase transition energies (in the spirit of the Cahn Hilliard energies for liquid-liquid phase transtions), known in the literature of dislocations as (generalised) Nabarro-Peirls models (see \cite{CoGarMu11} and the references therein). The natural representation of the line energy here is given by functionals defined on the space of $BV$ functions with values in a discrete group, so that dislocations are identified by the jump set of such functions, see \cite{ConGarMul16}.

In this paper we are interested in the case of a large number of dislocations on a macroscopic scale.  The interest of this case lying in the fact that a large quantity of dislocations is responsible for plastic deformation in the material (see \cite{ConGarMul17} for a more complete analysis). In particular, starting from the rescaled version of the line tension energy defined in \eqref{eq-Esigma}, we are interested in recovering an effective energy for a  large system of dislocations on a scale at which they can be seen as diffused. In this respect our limiting energy $E_0$ can be understood as the macroscopic (self) energy associated to a continuous distribution of dislocations. 

The result presented in here is also a crucial step for a derivation of a macroscopic model for plasticity accounting  for the presence of defects in the same spirit of \cite{GaLePo08}. The derivations of such macroscopic models as limit of elastic energies for incompatible fields, under proper energy scalings, will appear in a forthcoming paper.

The structure of the paper is as follows: in Section \ref{sec-preliminaries} we give preliminary definitions and recall some known results. In Section \ref{The Gamma-convergence result} we state the main result and present the proof of the lower bound.
In Section \ref{sec-approx} we present the approximation results needed for the upper bound, the latter being proved in Section \ref{sec-upper}.

\section{Preliminaries and statement of the main result}\label{sec-preliminaries}

We first set all the notation needed for the statement and proof of our main result. Unless further specified, in what follows $\Omega$ is a bounded open set of $\R^n$ with Lipschitz boundary.
In what follows, without loss of generality we will assume $\C B=\Z^N$, for  $N\geq 1$.

\subsection{Configurations}
We start with the set of admissible configurations.
In what follows we will denote with $[\C M(\Omega)]^{N\times n}$ the space of bounded Radon measures with values in $\R^{N\times n}$. 
The space $\C M^1(\Omega; \C S\times \B S^{n-1})$ will denote the set of 
 measures in  $[\C M(\Omega)]^{N\times n}$ of the form
\begin{equation}\label{form of dislocation}
    \mu=\theta\otimes \tau \ha^1\mres \gamma,
\end{equation}
where $\gamma\subset \BR^n $ is a  $1$-rectifiable set with tangent vector $\tau\in \mathbb{S}^{n-1}$  defined $\ha^1$ a.e.  on $\gamma$, $\C S \subset \BR^N$ is a generic subset, and $\theta\in [L^1(\gamma, \ha\mres\gamma)]^N $ is such that $\theta(x) \in  \C S$ for $\ha^1$-a.e. $x\in \gamma$.  We remark that in most cases we will consider 
the set $\C S$ to be a discrete lattice that spans $\B R^N$ but for notational purposes it is convenient to give the above definition for general sets. We also observe that although the main result, Theorem \ref{th-main}, is stated for the case $n=3$, in Section $3$ we prove an approximation result which holds in any dimension, hence it is convenient to give the relevant definitions for an arbitrary dimension $n\geq 2$.

We say that a measure in $[\C M(\Omega)]^{N\times n}$ is divergence free if it is row-wise divergence free, i.e., if the following holds true 
$$
\int_\Omega \sum_{j=1}^n\partial_j \varphi d\mu_{ij}=0\qquad \forall \varphi\in C_0^\infty(\Omega), \qquad \forall i=1,\cdots,N.
$$

The subset of  $\C M^1(\Omega; \C S\times \B S^{n-1})$ (and   $[\C M(\Omega)]^{N\times n}$) of divergence free measures will be denoted by $\C M_{df}^1(\Omega;\C S\times \B S^{n-1})$ (respectively   $[\C M(\Omega)]_{df}^{N\times n}$).  Note that, if  $G\in [L^1(\Omega)]^{N\times n}$ and $\mu= G \C L^n\in [\C M(\Omega)]^{N\times n}$ with $\C L^n$  the $n$-dimensional Lebesgue measure, then $\Div  G=0$ in the sense of distributions  if and only if $\mu\in \C [\C M(\Omega)]_{df}^{N\times n}$. 

We say that a measure  $\mu\in \C M^1(\Omega;\C S\times\B S^{n-1})$  is polyhedral if its support if formed by a finite number of straight closed segments.

\begin{remark}\label{obs meaning of divergence free constraint}
	Given a measure $\mu=b\otimes t\ha^1\mres\gamma\in \C M^1(\Omega;\Z^N\times \B S^{n-1})$, where $b\in \Z^N$ is constant and   $\gamma\in [\operatorname{Lip}([0,1])]^{n}$ is a curve such that $\gamma^\prime=t$, it holds that
	\begin{equation}\label{behaviour divergence for a one current}
		\langle \mu, \nabla \varphi\rangle = \langle b,\varphi(\gamma(1))-\varphi(\gamma(0))\rangle,\ \ \forall \varphi \in [C_c^1(\BR^n)]^N,
	\end{equation}
	hence it must be
	\begin{equation}
		\Div \mu = b\delta_{\gamma(0)}-b\delta_{\gamma(1)},
	\end{equation}
	where $b\delta_a\in [\C M(\Omega)]^{N}$ is a Dirac delta centered at $a\in \Omega$ with multiplicity $b$. Accordingly we say that $\mu$ carries a mass of $b$ at $\gamma(0)$ and a mass of $-b$ at $\gamma(1)$, where the sign depends on the orientation $t$.
	In particular for such a measure to be  divergence free in $\Omega$, the curve $\gamma$ must not have endpoints contained in $\Omega$.
	
\end{remark}
\begin{remark}\label{rem-structure}
If $\Omega$ is a simply connected domain, measures in $\C M_{df}^1(\Omega;\Z^N\times \B S^{n-1})$ can be extended to measures on the whole of $\R^n$ that can be characterized as measures concentrated on unions of countably many closed Lipschitz loops  with constant multiplicity in $\Z^N$ (see \cite{ConGarMas15}, Theorem 2.5, for the precise statement given in terms of $1$-rectifiable currents).
\end{remark}

The set  $\C M_{df}^1(\Omega;\Z^N\times \B S^{2})$ represents the set of admissible configurations for the class of energies under consideration.

\subsection{Energy densities and their main properties}
Here we recall the main properties of the class of energy densities 
$$
\psi:\Z^N\times \B S^2 \longrightarrow [0,+\infty).
$$
The $\ha^1$-elliptic envelope of $\psi$ is the function $\psi^{rel}$  obtained by solving, for any $b \in \mathbb{Z}^{N}$ and $t \in \mathbb{S}^{2}$, the cell problem
\begin{align}\label{definition H one elliptic envelope}
\psi^{\mathrm{rel}}(b, t):=\inf \Big\{&\int_{\gamma} \psi(\theta, \tau) d \mathcal{H}^{1}: \mu=\theta \otimes \tau \mathcal{H}^{1}\mres\gamma\in \C M_{df}^1(B_1;\Z^N\times \B S^{n-1}),\\
&
\operatorname{supp}\left(\mu-b \otimes t \mathcal{H}^{1}\mres\left(\mathbb{R} t \cap B_{1 / 2}\right)\right) \subset \subset B_{1 / 2}\Big\},\nonumber
\end{align}
where $B_{r}=B_{r}(0)$ denotes a ball of radius $r$ and center 0. We say that $\psi $ is $\ha^1$-elliptic if $\psi^{rel}=\psi$ (see \cite{ConGarMas15}).

 We will assume that $\psi$ is $\ha^1$-elliptic and we extend it to the whole of $\R^N\times\B S^2$  by setting $\psi(b, t)=+\infty$ for all $b \in \mathbb{R}^{N} \smallsetminus \Z^{N}$. Further, we assume that 
\begin{equation}
	\label{eq-lower-bound-psi}
	\psi(b, t) \geq c\,|b| \qquad \forall\ b\, \in \Z^{N} \smallsetminus \{0\} \hbox{ and } t \in \B S^{2}.
\end{equation}
Moreover we recall that $\ha^1$-ellipticity implies that $\psi$ is subattidive  and has linear growth at infinity in the first entry, i.e., for all $b,b'\in \R^N$ and $t\in \B S^2$ it satisfies
\begin{equation}
	\label{eq-helliptic-prop}
\psi(b+b',t)\leq	\psi(b,t)+\psi(b',t)\qquad \hbox{and }\qquad \psi(b,t)\leq \overline c \,|b|,
\end{equation}
for some positive constant $\overline c$,  see \cite[Lemma 3.2 (iii) and (iv)]{ConGarMas15}.

The recession function  $\psi_{\infty}: \mathbb{R}^{N} \times \B S^{2} \rightarrow$ $[0, \infty]$ of  $\psi$ is given by
\begin{equation}\label{definition of recession function}
    \psi_{\infty}(b, t):=\liminf _{s \rightarrow +\infty} \frac{1}{s} \psi(s b, t).
\end{equation}

This is a crucial ingredient in order to determine the effective energy density of the limiting energy $E_0$. For the readers' convenience we now recall some of the main properties of  $\psi_\infty$ as they are proved in \cite{ConGarMul17}.

\begin{proposition}\label{prop-recession-properties}
	Let $\psi:\R^N\times \B S^2\to \R\cup\{+\infty\}$ be $\ha^1$-elliptic, satifying \eqref{eq-lower-bound-psi} and such that $\psi(b,t)=+\infty$ if $b\in \R^N\setminus\Z^N$. Let $\psi_\infty$ be its recession function. Then the following hold:
	\begin{enumerate}
		\item $\psi_\infty(\cdot,t)$ is positively $1$-homogeneous for all $t\in \B S^2$;
		\item  Let $\mathcal{Q}=\left\{\lambda z: \lambda>0, z \in \mathbb{Z}^{N}\right\}$, then $\psi_{\infty}(b, t)=+\infty$ if $b \in \mathbb{R}^{N} \smallsetminus \mathcal{Q}$, and $\psi_{\infty}(b, t) \leq c\,|b|$ for $b \in \mathcal{Q}$;
		\item \label{lemma convergence psi to psi infty}
		Let $b \in \mathbb{R}^{N}, t \in \B S^{2}$, for any sequence $z_{j} \in \mathbb{Z}^{N}$ such that $\left|z_{j}\right| \rightarrow +\infty$ and $z_{j} /\left|z_{j}\right| \rightarrow b /|b|$ one has
		$$
		\lim _{j \rightarrow +\infty} \frac{1}{\left|z_{j}\right|} \psi\left(z_{j}, t\right)=\frac{1}{|b|} \psi_{\infty}(b, t).
		$$
	\end{enumerate}

\end{proposition}
The proof of property (i) is immediate, while properties (ii) and (iii) can be found in \cite{ConGarMul17}, Lemma 3.3 and 3.4 respectively.

Finally we define the function $g: \BR^{N\times 3}\longrightarrow [0,+\infty)$, which provides the effective energy density,  as the convex envelope of
\begin{equation}\label{eq-g-infty}
	g_\infty(A)=\begin{cases}\psi_\infty(b,t) \quad &\hbox{if } A=b\otimes t,\ \ b\in \BR^N,\ \  t\in \mathbb{S}^2,\\
		{+\infty} &\hbox{otherwise},
		\end{cases}
\end{equation}
i.e., $g(A)=g^{**}_\infty(A)$. Some important properties of $g$ are described in the following

\begin{lemma}\label{lemma properties g}
The function $g$ is continuous, $1$-homogeneous, and there are $c_1,c_2>0$ such that
$$
c_1 |A|\leq g(A) \leq c_2|A|,
$$
 for all matrices $A\in \R^{N\times 3}$.
\end{lemma}
\begin{proof}
Every matrix $A\in\R^{N\times 3}$ can be decomposed as a convex combination of rank 1 matrices on  which $g_\infty$ is finite, namely 
$$
A=\sum_{i=1}^N\sum_{j=1}^3 \dfrac{
|A_{ij}|}{3\|A e_j\|_1}\dfrac{3A_{ij}\|A e_j\|_1}{|A_{ij}|} e_j\otimes  e_i =\sum_{i=1}^N\sum_{j=1}^3 \lambda_j^i q_j^i\otimes  e_i,$$
where $\lambda_j^i:= 3^{-1}|A_{ij}|\|A e_j\|_1^{-1} $, $\sum_j\sum_{i}\lambda_j^i=1$, and  $q^i_j:=3A_{ij}\|A e_j\|_1|A_{ij}|^{-1} e_j\in \mathcal{Q}$. Then by convexity we have
$$
g(A)\leq \sum_{i,j}\dfrac{1}{3}\dfrac{|A_{ij}|}{\|A e_j\|_1} \psi_\infty\braket*{ q^i_j\otimes e_i}\leq \sum_{i,j}\dfrac{1}{3}\dfrac{|A_{ij}|}{\|A e_j\|_1} c\,|q^i_j|\leq c |A|.
$$ 
Being convex and finite, $g$ is continuous.
Now from the $1$-homogeneity of $\psi_\infty$ we infer that also $g_\infty$ is positively $1$-homogeneous. By Caratheodory's Theorem for every $\xi\in\BR^{N\times 3}$ and $ \eta >0$, there exist $3N+1$ vectors $\xi_i\in\BR^{N\times 3}$, and numbers $ t_i\geq0 $, such that $\sum_{i=1}^{3N+1}t_i=1$ and $ \xi=\sum_{i=1}^{3N+1} t_i \xi_i$ and $g(\xi)+\eta\geq \sum_{i=1}^{3N+1} t_i g_\infty(\xi_i)$; hence we have that 
\begin{equation*}
    \lambda g(\xi) +\lambda\eta\geq\lambda \sum_{i=1}^{3N+1} t_i g_\infty(\xi_i) =\sum_{i=1}^{3N+1} t_i g_\infty(\lambda \xi_i)\geq g(\lambda\xi),
\end{equation*}
and then $\lambda g(\xi)\geq g(\lambda\xi)$. The opposite inequality is finally obtained similarly replacing $\lambda$ and $\xi$ by  $\lambda^{-1}$ and $\lambda \xi$. The bound from below is a consequence of continuity and $1$-homogeneity.
\end{proof}

\subsection{The $\Gamma$-convergence result}\label{The Gamma-convergence result}

We now have all the ingredients in order to state the main result of the paper.

\begin{theorem}
	\label{th-main}
	Let $\psi:\Z^N\times \B S^2 \longrightarrow [0,+\infty)$ be $\ha^1$-elliptic and obey $\frac{1}{c}|b|\leq \psi(b,t)$ for all $b\in \Z^N$ and $t\in \B S^2$. Let $\Omega \subset \B R^3$ be an open bounded set, uniformly Lipschitz and simply connected. Then the functionals
	\begin{equation}
		E_\sigma(\mu):=\begin{cases}\displaystyle{\int_\gamma{\sigma \psi\left (\frac{\theta}{\sigma},\tau\right)d\ha^1}}& \textit{if }\mu=\theta\otimes \tau   \ha^1\mres\gamma \in \C M_{df}^1(\Omega; \sigma\Z^N\times\B S^2),\\
			+\infty & otherwise,
			\end{cases}
	\end{equation} 
	$\Gamma$-converge, as $\sigma \to 0$, with respect to the weak$*$ topology of $[\mes]^{N\times 3}$, to
	\begin{equation}
		E_0(\mu)=\begin{cases}\displaystyle{\int_\Omega{g\left(\frac{d\mu}{d\|\mu\|}\right)d\|\mu\|} }& \textit{if } \mu\in [\C M(\Omega)]_{df}^{N\times 3},\\
		+\infty& otherwise,
		\end{cases}
	\end{equation}
	where $g: \BR^{N\times 3}\longrightarrow [0,+\infty)$ is the convex envelope of $g_\infty$ as defined in \eqref{eq-g-infty}.
\end{theorem}

\begin{remark}
	Notice that from the lower bound on the density $\psi$ one immediately deduces that a sequence with equi-bounded energy has also bounded total variation. Therefore the compactness part of the $\Gamma$-convergence result is immediate.
\end{remark}

The proof of Theorem \ref{th-main} will be a consequence  of Proposition \ref{th-lower-bound} and Proposition \ref{pro-upper-bound} (respectively the lower and the upper bound).

As for the lower bound it is quite straightforward and it is a consequence of the definition of the energy density $g$. We give its proof with the proposition below.

\begin{proposition}
	\label{th-lower-bound}
		Let $\psi:\B Z^N\times \B S^2 \longrightarrow [0,+\infty)$ be $\ha^1$-elliptic and obey $\frac{1}{c}|b|\leq \psi(b,t)$ for all $b\in \B Z^N$ and $t\in \B S^2$, let $\Omega \subset \B R^3$ be  open and bounded. Then for every sequence $\sigma_j\to 0$ and $\mu_j\in \C M_{df}^1(\Omega;\sigma_j\Z^N\times\B S^2)$ converging weakly$*$ to some divergence free measure $\mu\in [\C M(\Omega)]^{N\times 3}$ we have
		$$
		\liminf_{j\to+\infty} E_{\sigma_j}(\mu_j)\geq E_0(\mu).
		$$
\end{proposition}

\begin{proof}
	
	By the subadditivity of $\psi$ we have that for all $ t\in \B S^2, b\in \B R^N, s >0, k\geq 1$ it holds
	\begin{equation*}
		\dfrac{\psi(s b,t)}{s }= \dfrac{k\psi(s b,t)}{sk}\geq\dfrac{\psi(ks b,t)}{s k},
	\end{equation*}
	hence, by the definitions of $\psi_\infty$ and $g$, we have
	\begin{equation}\label{eq-previous-ineq}
		\dfrac{\psi(s b,t)}{s }\geq\liminf_{k\rightarrow +\infty} \dfrac{\psi(ks b,t)}{s k}\geq \psi_\infty(b,t)\geq g(b\otimes t).
	\end{equation}
	Let $\mu_{j}\in\C M_{df}^1(\Omega;\sigma_j\Z^N\times\B S^2) $ be converging to $\mu \in [\C M(\Omega)]_{df}^{N\times 3}$ and be such that $\liminf_j E_{\sigma_j}(\mu_j)< +\infty $.
	Then from \eqref{eq-previous-ineq} we obtain
	\begin{equation*}
		E_{\sigma_j}(\mu_j)=\int_{\gamma_j} \sigma_j\psi(b_{j}\sigma_j^{-1},t_j)d\ha^1
		\geq \int_{\gamma_j} g(b_j\otimes t_j)d\ha^1=E_0(\mu_j).
	\end{equation*}
	To conclude we use the fact that $E_0$ is weakly lower semicontinuous by Reshetnyak's Theorem and Lemma $\ref{lemma properties g}$, so that
	\begin{equation}
		\liminf_{j\rightarrow +\infty}E_{\sigma_j}(\mu_j) \geq E_0(\mu).
	\end{equation}
	
\end{proof}

 The upper bound instead represents the core of the paper. It requires a technical construction and will be presented in the next section, where we will first show an approximation result for divergence free measures and then make an explicit construction with optimal energy.

\section{Approximation of divergence free measures}\label{sec-approx}

In this section we show two approximation results which are crucial for the  $\Gamma$-limsup inequality. We will show that any divergence free measure $\mu \in [\C M(\Omega)]^{N\times n}$ can be approximated, strictly  and therefore in energy, with measures that are absolutely continuous with respect to the Lebesgue measure,  piecewise constant and divergence free. Further we will show that, for $n=3$ the latters can be approximated with measures in $\C M_{df}^1(\Omega;\C Q\times \B S^{2})$. 
The case of dimension $n> 3$ presents some difficulties due to specific construction contained in Lemma \ref{lemma approximation from piecewise constant to straight segment }, we nevertheless expect that the approach followed in this paper can be adapted to the general case.

\subsection{Piecewise constant approximation}
Here we consider the general case of functions and measures in $\Omega\subseteq \R^n$, for $n\geq 2$.
First we introduce a class of admissible sequences of triangulations $\C T:=\{T^i\}_{i=1,\ldots, M}$ of  $\Omega\subseteq\R^n$.

\begin{definition}\label{def-triangulation}
We say that a family $\C T:=\{T^i\}_{i=1,\ldots, M}$ of simplexes  is an \emph{admissible sequence of triangulations (of aspect ratio $C_0$)} if the closed tetrahedra  $T^i$, with $1\leq i\leq M$, satisfy
\begin{enumerate}
	\item  $\Omega\subset\subset\bigcup_{i=1}^{M} T^i$;
	\item $\operatorname{int}(T^i)\cap \operatorname{int}(T^j)=\emptyset$ for $i\not =j$ (here $\operatorname{int}(T)$ is the topological interior part of the set $T$);
	\item there exists  a positive constant $C_0>0$ such that for  every $j$ there exists a point $x^j$ so that
	$$
	B_{ C_0 r}(x^j)\subseteq T^j\subseteq B_{ r}(x^j),
	$$
	where $r=\max_{i}{\operatorname{diam}  }(T_i)$ is the size of the triangulation.
\end{enumerate}
We say that a function $f$ is \emph{piecewise constant relatively to $\C T$} if 
$f$ is constant in $\operatorname{int}(T^j)$ for every $j\in\{1,\ldots, M\}$. 
\end{definition}

\begin{theorem}[Piecewise constant approximation]\label{prop Piecewise constant approximation}
	Let $\Omega$ be a simply connected open set with Lipschitz boundary, and let $\mu \in [\C M(\Omega)]^{N\times n}$ with $\Div \mu= 0$ be given. Then  there exists a sequence of measures $\mu_k\in [\C M(\Omega)]^{N\times n}$  such that $\mu_k\wsconv\mu$, with $\mu_k=A_k \C L^n$ and $A_k$ is piecewise constant relatively to a sequence of admissible triangulation $\C T_k$,    and $\Div \mu_k=0$ in $\Omega$. Furthermore $\lim_k \|\mu_k\|(\Omega)=\|\mu\|(\Omega)$.
\end{theorem}

The proof of Theorem \ref{prop Piecewise constant approximation} will be given essentially in  two steps. At first, in Lemma \ref{lemma convolution of measure whole domain}, we approximate $\mu$ via measures having smooth densities up to the boundary. In the second step  we reduce to measures which are piecewise constant with respect to a triangulation of $\Omega$.
In both cases the main difficulty is given by the free divergence constraint, and in order to modify the measures while preserving this constraint it will be convenient to interpret $\mu$ as a current, since the push forward of a current preserves solenoidality.

We start by regularizing the measures. The following lemma is an adaptation to the present context of Proposition 6, Chapter 5 in \cite{giaquinta1998cartesian}. 

\begin{lemma}[Smoothing]\label{lemma convolution of measure whole domain}
Let $\Omega\subset \BR^n$ be a bounded open Lipschitz set, let $\mu\in [\mes]^{N\times n}$ be divergence free, then it exists a family of functions $F_\eps\in [L^1(\Omega)]^{N\times n}\cap[C^\infty(\Omega)]^{N\times n}$ such that
\begin{itemize}
	\item[i)]  $F_\eps \C{L}^n \wsconv \mu$ in the sense of measures;
	\item[ii)] $\|F_\eps \C{L}^n \|(\Omega)\rightarrow \|\mu\|(\Omega)$;
	\item[iii)] $\Div F_\eps =0$ in $\Omega$.
\end{itemize} 
\end{lemma}
\begin{proof}
Let $d\in C^\infty(\Omega)$ be a smooth version of $d(\cdot,\Omega^c)$, namely a function satisfying  $0< d(x)< d(x,\Omega^c)$ and $|\nabla d(x)|\leq 1 $ for all $x\in\Omega$. We define, for every $z\in B_1(0)$, $H_z(x):=x+d(x)z$ and observe that $H_z(\Omega)=\Omega$,  $\nabla_x H_z(x)=Id+z\otimes \nabla d(x)$. Let $\rho_\eps$ be a convolution kernel and define the regularization of $\varphi \in [C_0(\Omega)]^{N\times n}$ for all $x \in \Omega$ as follows
\begin{align*}
     \varphi_\eps(x)&:=\int_{B_1(0)}\rho_\eps (-z)  \varphi(x+zd(x))\nabla_x H_z(x)dz\\
     &=\int_\Omega{d(x)^{-n}\rho_\eps\left(\frac{x-y}{d(x)}\right)\varphi (y)\left ( Id+\frac{y-x}{ d(x)}\otimes \nabla d(x)\right)dy},
\end{align*}
where we performed the change of variable $y=x+d(x)z$.
We then define by duality the mollification of $\mu$ to be $\langle \mu_\eps, \varphi\rangle := \langle \mu, \varphi_\eps\rangle$, hence
\begin{align*}
    \langle \mu_\eps, \varphi \rangle &=  \int_\Omega \left(\int_{B_1(0)}\rho_\eps (-z)  \varphi(x+zd(x))\nabla_x H_z(x)d z\right)d \mu(x)\\
    &=\int_\Omega \langle\int_\Omega  d(x)^{-n}\rho_\eps\left(\frac{x-y}{d(x)}\right)\varphi(y)\left ( Id+\frac{y-x}{ d(x)}\otimes \nabla d(x)\right)d y,G(x)\rangle d \|\mu\|(x),
\end{align*}
where $G\in [L^1(\Omega, \|\mu\|)]^{N\times n}$ is such that  $|G|=1$ $\|\mu\|$-a.e. and $G \|\mu\|=\mu$.

Rearranging the integrals in the definition of $\mu_\eps$ and using Fubini's Theorem it is easy to see that 
$\mu_\eps<\!\!<\C L^n$ with density function defined for $y\in \Omega$ by
\begin{equation}
        F_\eps(y):=\int_\Omega  d(x)^{-n}\rho_\eps\left(\frac{x-y}{d(x)}\right)G(x)(Id+  \nabla d(x)\otimes\frac{y-x}{d(x)} )d \|\mu\|(x).
    \end{equation}
Clearly $F_\eps \in [C^\infty(\Omega)]^{N\times n}$.

By uniform continuity,  $\varphi_\eps$ converges uniformly in $\Omega $ to $\varphi$, and thus $\mu_\eps \wsconv \mu$ in $[\mes]^{N\times n}$, which proves i). Furthermore it holds
\begin{align*}
|\langle
\mu_{\eps},\varphi\rangle|&=\left| \int_\Omega\langle\int_{B_\eps(0)} \rho_{\eps}(-z)(Id+\nabla d(x) \otimes z)\varphi(x+zd(x))d z,G(x)\rangle d \|\mu\|(x)\right|\\
    &\leq\int_\Omega (1+2\eps)\|\varphi\|_\infty d \|\mu\|(x),
\end{align*}
hence 
\begin{equation}
    \|\mu_{\eps}\|(\Omega)\leq (1+2\eps)\|\mu\|(\Omega),
\end{equation}
thus $\lim_{\eps\rightarrow0}\|\mu_{\eps}\|(\Omega)=\|\mu\|(\Omega)$, and therefore ii) holds.

Finally we now prove that $\mu_\eps$ is  row-wise divergence free, i.e.,
\begin{equation}
   \sum_{j=1}^n\int_\Omega F_\eps^{ij}(y)\partial_j \psi(y) d y=0 , \ \ \ \ \forall \psi \in C_c^1(\Omega), \ \ i=1,\cdots, N.
\end{equation}
To see this, we denote with $G^i$ the $i$-th row of $G$, then we compute
\begin{align*}
    &\sum_{j=1}^n\int_\Omega F_\eps^{ij}(y)\partial_j \psi(y)d y\\
    &=\sum_{j,l=1}^n\int_{B(0,1)} \int_\Omega  \rho_\eps(z)G_{il}(x)\partial_l (H_z)^j(x) \partial_j \psi(x+d(x)z) d \|\mu\|(x) d z\\
    &= \int_{B(0,1)} \rho_\eps(z)\int_\Omega   \langle G^i,\nabla (\psi \circ H_z)(x)\rangle  d \|\mu\|(x)d z,
\end{align*}
where we used the change of variables $y=x+d(x)z$ and, in the last equality, we used the fact that 
$$
\sum_{j=1}^n \partial_l (H_z)^j(x) \partial_j \psi(x+d(x)z) = \partial_l(\psi \circ H_z)(x) .
$$
 The claim now follows since $\psi \circ H_z \in C_c^1(\Omega) $ and $\mu$ is divergence free.

\end{proof}
We say that a mapping $\phi$ is a potential for a divergence free matrix field $F$ if $\C R \phi =F$ for some first order linear differential operator $\C R$ satisfying $\operatorname{div} \C R\psi = 0$ for every $\psi \in [C^\infty(\Omega)]^{N\times m}$ where $m=n(n-1)/2$. For example if $n=3$ then $\C R=\operatorname {curl}$. We observe that if $\Omega$ is simply connected then such an operator $\C R$ always exists as a direct consequence of Poincaré's Lemma.

Thanks to Lemma \ref{lemma convolution of measure whole domain}, we simply need to prove Theorem \ref{prop Piecewise constant approximation} for divergence free measures of the form $\mu=F\C L^n$. To do so we would like to pass to a potential $\phi$ of $F$, and then approximate $\phi $ with piecewise affine functions by linear interpolating over a sequence of triangulations of $\Omega$.
In order to show the convergence of the interpolating sequence to $\phi$ we need boundedness of its derivatives (see for instance \cite{LeoniSobolev}), while from Lemma \ref{lemma convolution of measure whole domain} we can only infer  $\phi\in[C^\infty(\Omega)]^{N\times m}$. To obtain such bound we will modify $F$, since clearly a uniform bound for the  derivatives of $F$ implies bounds for the derivatives of $\phi$. With this goal in mind we state below an extension lemma  for $F$ whose proof follows closely the one of a similar extension lemma proved in  \cite[Lemma 2.3]{ConGarMas15} in the context of $1$-rectifiable currents.

\begin{lemma}[Extension]\label{lemma extension}
Let $\Omega\subset\BR^n$ be a  bounded Lipschitz set. There exist an open bounded set $\hat\Omega$ compactly containing $\Omega$ such that for every $F\in [L^1(\Omega)]^{N\times n}$ with $\Div  F=0$ in $\Omega$, there is a function $\hat F\in [L^1(\hat\Omega)]^{N\times n}$, with $\Div  \hat F=0$ in $\hat\Omega$, such that $\hat F=F$ in $\Omega$.
In particular the measure $\hat{\mu}=\hat F \C L^n\in [\C M(\hat\Omega)]_{df}^{N\times n}$ extends the measure $\mu=F \C L^n$.
\end{lemma}

\begin{proof}
We will prove the lemma, first in the case of $F\in [L^1(\Omega)]^{n}$ with $\Div F=0$ in $\Omega$, then the matrix valued case will follow simply by extending row-wise.\\
Choose a function $N \in C^1 (\partial\Omega; \mathbb{S}^{n-1} )$, such that $N (x) \cdot \nu(x) \geq \alpha > 0$ for
almost all $x \in \partial \Omega$, where $\nu$ is the outer normal to $\partial \Omega$ (see \cite{Ho07} for details).
Consider the mapping $g : \partial\Omega \times (-\rho_0, \rho_0) \rightarrow \BR^n$ defined by $g(x, t) = x + tN (x)$, then there exists $\rho_0$ sufficiently small such that  $g$ is bijective and bi-Lipschitz: indeed, by Local Invertibility Theorem for Lipschitz mappings (see \cite{Cl83}), there exists $\rho_0$ small enough such that $g$ is locally invertible  with Lispchitz inverse function, furthermore, since $g(x ,0) =x$, we also get global invertibility (see \cite{belnovpao01} Lemma 33). Let $D_0 = g(\partial\Omega\times (-\rho_0, \rho_0))$ and $h : D_0\rightarrow D_0$
be defined by $h (g(x,t))= g(x,-t)$. Then $h$ is bi-Lipschitz and coincides with
its inverse. We then set $\hat \Omega:=\Omega\cup D_0$ and we define the extension 
$\hat{F}(y)=\chi_{\Omega}F(y)+\chi_{D_0 \smallsetminus \Omega}\nabla h(h^{-1}(y))F(h^{-1}(y))|\nabla h^{-1}|$, for all $y\in \hat\Omega$.

We now show that $\hat F$ is divergence free. We compute, for every $\psi \in C_c^1(\hat\Omega)$, 
\begin{equation*}
	\begin{split}
	\int_{\hat \Omega}\langle \nabla \psi , \hat F \rangle dy &=\int_\Omega\langle  \nabla \psi ,F\rangle dy +
	\int_{D_0\setminus\Omega}\langle \nabla \psi ,  \nabla h(h^{-1}(y))F(h^{-1}(y))\rangle |\nabla h^{-1}| dy\\
	&= \int_\Omega\langle  \nabla \psi ,F \rangle dy +
\int_{\Omega\cap D_0}\langle \nabla[\psi\circ h](x), F(x)\rangle dx.
	\end{split}
\end{equation*}
%
We then  observe that, by direct computation, it holds  
$$
D[(\chi_{D_0 \smallsetminus \Omega}\psi)\circ h]= \chi_{\Omega\cap  D_0}(x)\nabla[\psi\circ h] ,
$$ 
in the sense of distributions in $\Omega$, thus in particular $ (\chi_{D_0 \smallsetminus \Omega}\psi)\circ h \in W^{1,\infty}(\Omega)$. Therefore we obtain  
\begin{equation*}
    \int_{\hat \Omega}\langle \nabla \psi , \hat F \rangle dy=\int_\Omega\langle  \nabla [\psi-(\chi_{D_0\setminus \Omega}\psi)\circ h] ,F \rangle dy.
\end{equation*}
The right hand side  is zero since $\varphi=\psi -(\chi_{D_{0} \smallsetminus \Omega}\psi)\circ h\in W^{1,\infty}_0(\Omega)$: indeed it is  zero for all $\varphi\in C^\infty_0(\Omega)$ and then also in $W^{1,\infty}_0(\Omega)$ by  density of smooth and compactly supported functions with respect to the weak$*$ topology in $W^{1,\infty}(\Omega)$.

The general case is obtained by using the above construction row-wise.
\end{proof}

We are ready to prove Theorem \ref{prop Piecewise constant approximation}.

\begin{proof}[Proof of Theorem \ref{prop Piecewise constant approximation}]
Thanks to Lemma \ref{lemma convolution of measure whole domain}, we first find a sequence of fields $F_h\in [L^1(\Omega)]^{N\times n}\cap[C^\infty(\Omega)]^{N\times n}$ such that  $\mu_h=F_h\C L^n \wsconv\mu$, $\Div F_h=0$ in $\Omega$ and $\lim_{h\to+\infty}\|\mu_h\|(\Omega)=\|\mu\|(\Omega)$.

Using Lemma \ref{lemma extension} we can find a set $\hat\Omega $ compactly containing $\Omega$ and a sequence of functions $\hat F_h$ extending $F_h$ such that $\Div \hat F_h=0 $ in $\hat \Omega$.
We then choose $\Omega'$ such that $\Omega \subset\subset\Omega^\prime\subset\subset\hat\Omega$ and by (standard) convolution we can now  find a sequence $G_h\in [L^1(\Omega^\prime)]^{N\times n}\cap [C^\infty(\Omega^\prime)]^{N\times n}$, with $\Div G_h =0 $ in $\Omega^\prime$, such that $G_h \C L^n $ converges strictly to $\mu$.
We then apply Poincaré Lemma to $G_h$ and find a potential $\phi^h\in [C^\infty(\Omega^\prime)]^{N\times n(n-1)/2}$, such that $\C R \phi^h =G_h$.

Now we fix a sequence of admissible triangulations 
$$\C T_k:=\{T_k^i\}_{i=1,\ldots, M(k)},
$$
of aspect ratio $C_0$ and size $1/k$, such that $ \Omega\subset\subset \bigcup_{i=1}^{M(k)} T_k^i \subset \subset \Omega^\prime$, and for every $h$ we construct a sequence of piecewise affine functions $\phi^h_k$ obtained interpolating linearly the values of $\phi^h$ on the vertices of the tetrahedra $T_k^i$. By classical discretization arguments (see for instance \cite{LeoniSobolev}, Theorem 11.40) we have that
 $$
 \|\nabla \phi^h_k-\nabla \phi^h\|_{L^1(\Omega)}\leq \frac{C}{k^2}\sup_{\Omega} |\nabla^2\phi^h |,
 $$
where the constant $C$ depends on $C_0$.
  Hence $G^h_k:=\C R\phi^h_k$ converges  to $G_h$ in $L^1$ for every $h$. Furthermore $\Div  G^h_k=0$ in $\Omega$ and $G^h_k$ is piecewise constant.
  
  In conclusion by a diagonal procedure we obtain a sequence of piecewise constant divergence free fields that converge strictly to $\mu$ in $\Omega$.
\end{proof}

\subsection{Optimal construction via polygonal supported measures}
From now on we focus on the special case $n=3$. Indeed the constructions we perform in Lemma \ref{lemma recovery on a simplex 2 } and Lemma \ref{lemma approximation from piecewise constant to straight segment  } depend on the dimension: while in the case $n=3$ the shared face of two simplex is $2$ dimensional, in the  generic case two neighbouring simplex share a  $n-1$ dimensional face. Nevertheless we expect that our construction can be adapted to any dimension.

We now show a second approximation result that is closely related with our energies. We will need to show that any piecewise constant divergence free measure can be obtained as a limit of a sequence of measures (concentrated on lines) with equi-bounded energy. This density result requires a rather technical construction, a by-product of which is the theorem stated below and proved at the end of this section. 

\begin{theorem}\label{theo approximation straight lines}
	Given a divergence free measure $\mu \in [\C M(\Omega)]^{N\times 3}$, there exists a sequence of polyhedral measures $\mu_k\in \C M^1_{df}(\Omega;\C Q\times \B S^{2})$ such that $\mu_k\wsconv\mu$.
\end{theorem}

A result of this type can actually be obtained as a consequence of  the celebrated result of Smirnov \cite{Smir} which shows that every normal current without boundary in $\BR^n$ can be decomposed in elementary solenoids. As Bourgain and Brezis pointed out in \cite{BouBre}, this decomposition implies an approximation  for divergence free vector fields in terms of measures supported on curves. The proof of such approximation was given in \cite{GoHeSp}, where the authors show the existence of the approximating sequence by means of an argument that doesn't allow to choose the curves in the approximation. This feature clashes with our need to control the energy of the approximating sequence. Our result is instead constructive (see Lemma \ref{lemma recovery on a simplex 2 }) and will imply the $\Gamma$-limsup inequality in our main result.

More precisely we approximate $\mu$ by piecewise constant fields using Theorem \ref{prop Piecewise constant approximation}, then on  each tetrahedron of the triangulation we construct measures in $\C M_{df}^1(\Omega, \C Q\times \mathbb S^2)$ and then we glue these local approximations obtaining the result. This is the most delicate passage of the construction: indeed
gluing while preserving the divergence free constraint presents some difficulities, to overcome which it is important to choose the right boundary condition on each tetrahedron, see \ref{third item lemma on simplex} of Lemma \ref{lemma recovery on a simplex 2 } and Remark \ref{Remark explanation modification mass}.

We first start with  a single tetrahedron. To this aim we need to introduce some notation.

Given a tetrahedron $T\subset \BR^3$, we perform the following subdivision of its boundary $\partial T$ in closed triangles: consider
a face of $T$ with edges of length $l_1 \leq l_2 \leq l_3$, we divide each of these edges in $k$ segments of length ${l_i}/{k}$, $i\in\{1,2,3\}$,  and consequently we obtain a division of that face in $k^2$ (closed) triangles, denoted by $\Delta(h,k)$, with $h=1,\ldots, 4 k^2$, see Figure \ref{fig subdivision simplex}. Hence
\begin{equation}\label{def-division in subsimplexes}
    \partial T=\bigcup_{h=1}^{4k^2}\Delta(k,h).
\end{equation} 
 Note that there exists a universal constant $C>0$ such that  for all $(h,k)$ we have
	\begin{equation}
		\label{eq-measure-triangle}
		\mathcal{H}^2(\Delta(h,k))\leq C \frac{({\rm diam}(T))^2}{k^2}.
\end{equation}
We denote with $d(T,k,h)$ the baricenter of each triangle $\Delta(k,h)$, and we call $n_h$  the outer unit normal, with respect to $T$, in $\Delta(k,h)$.

\begin{figure}[t]
		\fontsize{8}{4}{
			\def\svgwidth{200pt}
		\includegraphics[width=6cm]{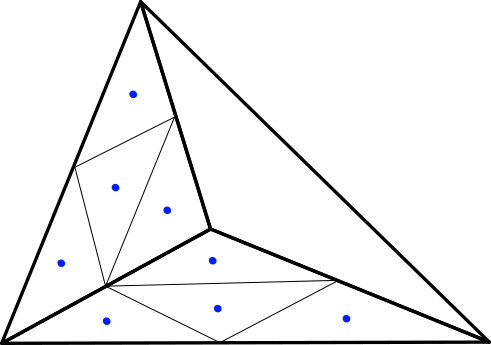}
   \caption[]{An example of the subdivision of the tetrahedron $T$, with two faces explicitly subdivided for $k=2$, where the blue dots are the baricenters of the
triangles obtained with the subdivision. }
			\label{fig subdivision simplex}
		}
	\end{figure}

\begin{lemma}\label{lemma recovery on a simplex 2 }
	Let $T\subset \BR^3$ be a $3$-simplex and $\C F$ a finite union of planes in $\BR^3$. Let $A\in \BR^{N\times 3}$ and assume that $A=\sum_{j=1}^M b^j\otimes t^j$, with $b^j \in \C Q$,  $t^j\in \mathbb{S}^2$. Then there exist sequences of polyhedral measures $\mu_k, \nu_k, \omega_k, \rho_k \in \C M^1_{{\rm loc}}(\R^3;\C Q\times\B S^2)$ such that $	\mu_k=\nu_k+\omega_k+\rho_k$ and  satisfy the following:
	\begin{enumerate}
		\item\label{zero item lemma on simplex}
		\begin{equation*}
		    \nu_k=\frac{1}{k^4}\sum_{j=1}^M  b^j\otimes t^j \ha^1\mres (\Gamma^j_k\cap T^k),
		\end{equation*}
 where  $\Gamma^j_k$ are straight lines parallel to $t^j$, satisfying 
	\begin{equation}
		\label{lemma-item-1-weak-convergence}
		\frac{1}{k^4}\ha^1\mres \Gamma^j_k\wsconv \C L^3,
	\end{equation}
and  $T_k\subset \subset T$ is a tetrahedron satisfying  ${\rm dist}(T_k,\partial T)\leq {\rm diam}(T)/k^2$; 
\item \label{ii of lemma simplex}$\mu_k\mres T=\nu_k+\omega_k$, $\mu_k\mres T^c=\rho_k$, $\|\mu_k\|(\partial T)=0$, $\ha^1(\operatorname{supp}\rho_k\cap \C F )=0$;
\item \label{iii of lemma simplex} for every compact set $K\subset \BR^3$ it holds $\lim_{k\to +\infty}\|\omega_k\|(K)+\|\rho_k\|(K)=0$;

		\item\label{first item lemma on simplex} for every $\varphi \in [ C_c(\BR^3)]^{N\times 3}$ it holds
		\[
		\lim_{k\rightarrow\infty}\int_T \varphi d \mu_k = \int_T\langle \varphi, A\rangle d x;
		\]
	
		\item \label{third item lemma on simplex} for every $\varphi \in [C_c^\infty(\BR^3)]^N$ it holds 
		$$
		\langle \mu_k, \nabla\varphi\rangle = \sum_{h=1}^{4k^2} \langle A n_h ,\varphi(d(T,k,h))\rangle \ha^2(\Delta(k,h)).
		$$
		
		 \end{enumerate}
\end{lemma}

\begin{proof}
Without loss of generality we may assume that the baricenter of $T$ coincides with the origin and we define $T_k:= (1-\frac{1}{k^2})T$. Therefore $T_k$ is a tetrahedron similar to $T$ and ${\rm dist}(T_k,\partial T)\leq {\rm diam}(T)/k^2$.

	The construction is quite natural but  somewhat involved, therefore we shall present it in several steps.
	
	\smallskip
	\emph{Step 1. The segments inside $T_k$.}

We  define the approximating measures inside $T_k$.
 	For every $j\in\{1,\ldots, M\}$ we consider the $2$-dimensional vector space $\Pi^j$ whose normal is $t^j$. Let $v_1^j,v_2^j$ be an orthonormal base of this plane, and consider the square lattice on $\Pi^j$ defined by $\mathcal{G}^j_k=\operatorname{span}_{\mathbb{Z}}(\frac{1}{k^2}v_1^j,\frac{1}{k^2}v_2^j)$. We then  set $\Gamma_k^j:=\mathcal{G}^j_k+ \BR t^j$ and define
\begin{equation}
    \nu_k^j:=\frac{b^j}{k^4}\otimes t^j
 \ha^1\mres \Gamma_k^j.
\end{equation}
It is easy to check that $\nu_k^j\mres T_k\in [\C M(\BR^3)]^{N\times 3}$  and that for all $\varphi \in [C_c(\BR^3)]^{N\times 3}$ it holds
\begin{align}\label{convergence against bounded functions}
    \lim_{k\rightarrow +\infty}\langle \nu_k^j\mres T_k, \varphi\rangle =\lim_{k\rightarrow +\infty}\langle \nu_k^j\mres T, \varphi\rangle&=\langle b^j\otimes t^j\C{L}^3\mres T, \varphi \rangle .
\end{align}
The measure inside $T_k$ is then
\begin{equation}\label{eq-notation-lemma-tetraedro-1}
	\nu_k:=\sum_{j=1}^M\nu_k^j\mres T_k,
\end{equation}
which then converges weakly$*$ to $A\, \C L^3\mres T$.

We now subdivide $\partial T$ in triangles that we call ${\Delta(h,k)} $, for $h=1,\cdots ,4k^2$ as specified in \eqref{def-division in subsimplexes}, i.e., $ \partial T=\bigcup_{h=1}^{4k^2}\Delta(k,h)$ and $d(T,k,h)$ is the baricenter of $\Delta(k,h)$. This subdivision induces in turns a subdivision of $\partial T_k$ in triangles $\delta(h,k)=(1-\frac{1}{k^2})\Delta(h,k)$ by projecting from the origin the $\Delta(k,h)$ onto $\partial T_k$.  Hence we can write $ \partial T_k=\bigcup_{h=1}^{4k^2}\delta(k,h)$ with $|x-y|\less C\operatorname{diam}(T)k^{-1}$ for every $x\in \Delta(h,k), y\in\delta(k,h)$.

Up to removing a negligible number of lines (so that \eqref{convergence against bounded functions} still holds),  we can assume that  $\Gamma^j_k$ intersects $\partial T_k$ only on isolated points and that none of these points belong to more than one triangle $\delta(k,h)$. Indeed for each one of the $4k^2$ triangles  $\delta(k,h)$ there are at most $\C O(k)$ lines that intersect its contour, hence there are at most $\C O(k^3)$ lines that ought to be discarded but each line has a mass of $\C O(k^{-4})$, hence the total error is negligible in the limit.  We can also assume that $\ha^1(\Gamma_k^j \cap \Gamma_k^{i})=0 $ for every $j\not=i$.

\smallskip
\emph{Step 2. Definition in $ T\smallsetminus T_k$}

For each $1\leq h\leq 4k^2$ we now want to connect the lines of $\nu_k$, which end on $\delta(k,h)$, with the baricenter of $\Delta(k,h)$. We define
\begin{equation}\label{eq-notation-lemma-tetraedro}
R(k,j,h):=\Gamma_k^j \cap \delta(k,h),\qquad N(k,j,h):=\#R(k,j,h).
\end{equation}
 We observe that  $\ha^2(\delta(k,h))=(1-\frac1{k^2})^2\ha^2(\Delta(k,h))$ and hence
\begin{equation}
	\label{eq-smaltriangle-vs-largetriangle}
	\begin{split}
&	|\ha^2(\Delta(k,h))-\ha^2(\delta(k,h))|\leq C\frac{(\operatorname{diam}(T))^2}{k^4},\\
&	 N(k,j,h)\leq Ck^4\ha^2(\delta(k,h)) \leq C(\operatorname{diam}(T))^2k^2.
	\end{split}
\end{equation}

 In order to concentrate the mass in the baricenter $d(T,k,h)$ of each $\Delta(k,h)$ we  connect each point $p\in R(k,j,h)$ to $d(T,k,h)$ using a small straight segment $[p,d(T,k,h)]$. On each of these segments  we then define the measure 
 \begin{equation}\label{eq-lemma-step2-omega}
\omega(k,j,h,p):=\frac{1}{k^4}b^j\otimes t_p 
\ha^1\mres [p,d(T,k,h)],
 \end{equation}
where $t_p$ is the unit tangent vector in the direction $(d(T,k,h)-p)\operatorname{sign}(\langle t_j,n_h\rangle)$ and $n_h$ is the outward normal vector of $\delta(k,h)$. 

We then define  
\begin{equation}
	\label{eq-measure-sulle-facce}
	\omega_k:= \sum_{j=1}^M\sum_{ h=1}^{4k^2}\sum_{p\in R(k,j,h)}\omega(k,j,h,p).
\end{equation}

Since $\|\omega(k,j,h,p)\|(T)\leq C\operatorname{diam}(T)|b^j|k^{-5}$ from \eqref{eq-smaltriangle-vs-largetriangle} we infer
\begin{equation}\label{eq-vanishing of segments1}
\|\omega_k\|(\R^n)\leq 	C\frac{1}{k}(\operatorname{diam}(T))^3\sum_{j=1}^M |b^j|,
\end{equation}
and then
\begin{equation}\label{eq-vanishing of segments2}
	\lim_{k\rightarrow+\infty}\|\omega_k\|(\R^n)=0.
\end{equation}
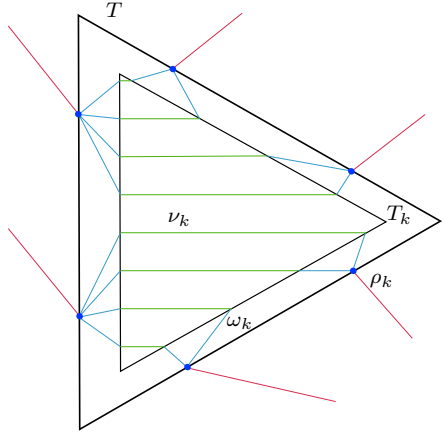
\begin{figure}[t]
		\fontsize{8}{4}{
			\def\svgwidth{300pt}
\begingroup%
  \makeatletter%
  \providecommand\color[2][]{%
    \errmessage{(Inkscape) Color is used for the text in Inkscape, but the package 'color.sty' is not loaded}%
    \renewcommand\color[2][]{}%
  }%
  \providecommand\transparent[1]{%
    \errmessage{(Inkscape) Transparency is used (non-zero) for the text in Inkscape, but the package 'transparent.sty' is not loaded}%
    \renewcommand\transparent[1]{}%
  }%
  \providecommand\rotatebox[2]{#2}%
  \newcommand*\fsize{\dimexpr\f@size pt\relax}%
  \newcommand*\lineheight[1]{\fontsize{\fsize}{#1\fsize}\selectfont}%
  \ifx\svgwidth\undefined%
    \setlength{\unitlength}{706.46008589bp}%
    \ifx\svgscale\undefined%
      \relax%
    \else%
      \setlength{\unitlength}{\unitlength * \real{\svgscale}}%
    \fi%
  \else%
    \setlength{\unitlength}{\svgwidth}%
  \fi%
  \global\let\svgwidth\undefined%
  \global\let\svgscale\undefined%
  \makeatother%
  \begin{picture}(1,0.58042518)%
    \lineheight{1}%
    \setlength\tabcolsep{0pt}%
    \put(0,0){\includegraphics[width=\unitlength,page=1]{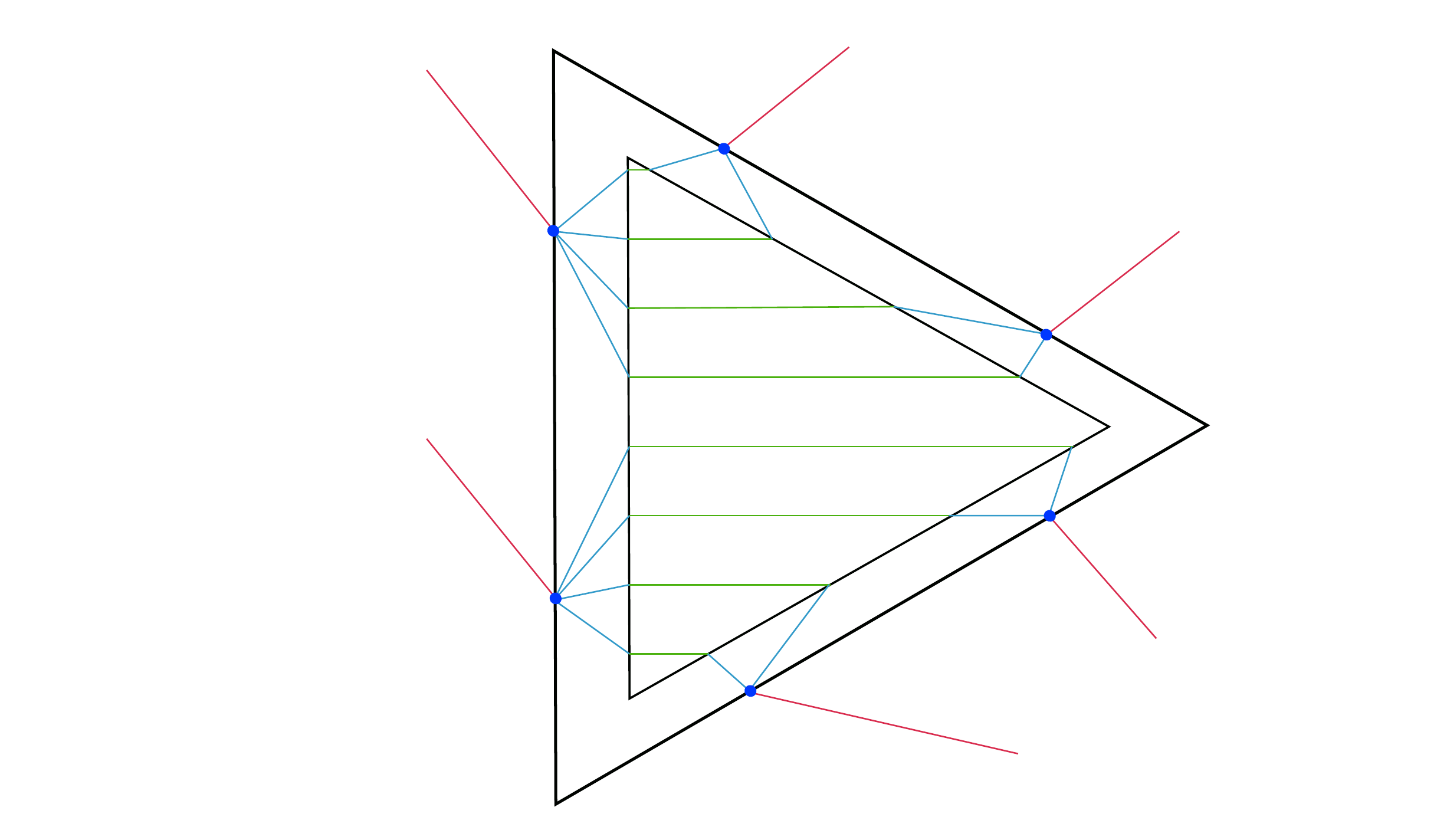}}%
    \put(0.49398862,0.28492236){\makebox(0,0)[lt]{\lineheight{0.40000001}\smash{\begin{tabular}[t]{l}$\nu_k$\end{tabular}}}}%
    \put(0.56787536,0.15507957){\makebox(0,0)[lt]{\lineheight{0.40000001}\smash{\begin{tabular}[t]{l}$\omega_k$\end{tabular}}}}%
    \put(0.7460547,0.20816745){\makebox(0,0)[lt]{\lineheight{0.40000001}\smash{\begin{tabular}[t]{l}$\rho_k$\end{tabular}}}}%
    \put(0.7672874,0.28778978){\makebox(0,0)[lt]{\lineheight{0.40000001}\smash{\begin{tabular}[t]{l}$T_k$\end{tabular}}}}%
    \put(0.41694914,0.54258125){\makebox(0,0)[lt]{\lineheight{0.40000001}\smash{\begin{tabular}[t]{l}$T$\end{tabular}}}}%
  \end{picture}%
\endgroup%

   \caption[]{A portion of $\mu_k$, for a single $j$ and in a two dimensional schematization. The blue dots are the baricenter of the $\Delta(k,h)'s$.}\label{fig esempio e scheletro misura}}
   \end{figure}

See Figure \ref{fig esempio e scheletro misura} for a representation of a portion of $\nu^j_k$ and $\omega(k,j,h,p)$.

\smallskip
\emph{Step 3. Definition outside $T$.}\\
We define the mass at the baricenter $d(T,k,h)$ to be the following vector valued quantity:
	\begin{equation}\label{definition exact mass}
	B(T,k,h):=\sum_{j=1}^{M} N(k,j,h)\operatorname{sign}(\langle t^j, n_h\rangle )\frac{b^j}{k^4}.
\end{equation}

This definition makes sense: indeed we observe that $\operatorname{supp}(\nu_k+\omega_k)$ is given by a finite family of piecewise straight lines connecting different baricenters, hence using Remark \ref{obs meaning of divergence free constraint} one can show that for every $\varphi\in [C_c^1(\BR^3)]^{N}$ it holds
\begin{equation}\label{Lemma tetra-identità-massa sui bordi}
    \langle \nu_k +\omega_k , \nabla \varphi \rangle = \sum_{h=1}^{4k^2}\langle B(T,k,h),\varphi(d(T, k,h))\rangle,
\end{equation}
which means precisely that the vector valued mass carried by $\nu_k+\omega_k$ at each baricenter on $\partial T$ is given by $B(T,k,h)$.

For every $\Delta(k,h)$ with $1\leq h \leq 4k^2$ we define the measures
 \begin{equation}
\rho_k(h):=(\ha^2(\Delta(k,h))A n_h -B(T,k,h))\otimes\tau(k,h)  \ha^1\mres \gamma(k,h),
 \end{equation}
where $\gamma(k,h)$ is an arbitrary half line with direction $\tau(k,h)$, not intersecting $\C F$ and having endpoint in $d(T,k,h)$ (see Remark \ref{Remark explanation modification mass} for the heuristics).\\
Again recalling Remark  \ref{obs meaning of divergence free constraint}, for every $\varphi \in [C_c^\infty(\BR^3)]^N$ it holds
\begin{equation}\label{eq-div-rho-k}
	\langle \rho_k (h), \nabla \varphi \rangle = \langle \ha^2(\Delta(k,h))A n_h -B(T,k,h),\varphi(d(T, k,h))\rangle.
\end{equation}

We then define the total average error as follows
\begin{equation}\label{eq-def-eq-rho-k}
\rho_k:=\sum_{h=1}^{4k^2}\rho_k(h).
\end{equation} 
We now want to show that,  for every compact set $K$,
$\|\rho_k\|(K)$ tends to zero. With that aim in mind we first prove the following claim.

\emph{Claim: there exists a universal constant $C$ such that
	\begin{equation}\label{inequality bound error masses}
		|B(T,k,h)-\ha^2(\delta(k,h))A n_h|\leq C\frac{1}{k^3}\operatorname{diam}(T)\sum_{j=1}^M|b_j|.
\end{equation}
}

From the definition of $B(T,k,h)$ and from the fact that $An_h=\sum_{j=1}^M b_j\langle t_j,n_h\rangle$ it is clear that we only need to consider $1\leq h\leq 4k^2$ and $1\leq j\leq M$ such that $\langle t^j,n_h\rangle  \not = 0$. Let then $j,h$  be as above and consider the elementary cell of the lattice $\C{G}_k^j$, i.e., $\C{C}_k^j:=\Pi_j \cap  \{\frac{s_1}{k^2}v_1^j+\frac{s_2}{k^2}v_2^j\, :\, 0\leq s_1,s_2\leq 1\}$. Let also $\Sigma_h$ be the plane that contains $\delta(k,h)$, then if we denote by  $P^{h,j}_k$ the elementary cell of the (planar) lattice $\Sigma_h \cap \Gamma^j_k$  we have that 
\begin{equation} \label{measure of the parallelogram}
	\ha^2(P^{h,j}_k)=\ha^2(\C{C}_k^j)\dfrac{1}{|\langle t^j, n_h\rangle |}=\frac{1}{k^4|\langle t^j, n_h\rangle |},
\end{equation}
since $\C{C}_k^j$ is obtained by orthogonal projection of a translation of $P^{h,j}_k$ on $\Pi_j$.\\
From the previous equality and the fact that  $A=\sum_{j=1}^M b_j\otimes t_j$ we obtain that
\begin{equation}\label{charcterisation approximated total mass}
	\ha^2(\delta(k,h))An_h=\sum_{j=1}^{M}\dfrac{\ha^2(\delta(k,h))}{\ha^2(P^{h,j}_k)}\operatorname{sign}(\langle t^j, n_h\rangle)\frac{b^j}{k^4} .
\end{equation}
Therefore in order to show \eqref{inequality bound error masses}, in view of the definition of $B(T, k,h)$, it is enough to show that
\begin{equation}\label{estimate difference btw number points and fraction of area}
	\left|\dfrac{\ha^2(\delta(k,h))}{\ha^2(P^{h,j}_k)}-N(k,h,j)\right|\leq C\operatorname{diam}(T)k.
\end{equation}
This inequality clearly holds true, since $\ha^2(\delta(k,h))/\ha^2(P^{h,j}_k)$ counts the number of points in $R(k,h,j)$ up to an error due to those points of the lattice  $\Sigma_h \cap \Gamma^j_k$ that  are close to the contour of $\delta(k,h)$. The number of such points can in turn be estimated by $C\operatorname{diam}(T)k$, which proves \eqref{estimate difference btw number points and fraction of area}.


Finally from \eqref{inequality bound error masses}, \eqref{eq-smaltriangle-vs-largetriangle} and  \eqref{eq-def-eq-rho-k} we obtain that for every compact set $K\subset\R^n$ we have
\begin{equation}\label{eq-rho-tende a zero}
\|\rho_k\|(K)\leq \sum_{h=1}^{ 4k^2} \|\rho_k(h)\|(K)\leq \frac{C}{k}\operatorname{diam}(T){\rm diam}(K)\sum_{j=1}^M|b_j|,
\end{equation}
which then tends to zero as $k\to +\infty$.

\smallskip

\emph{Step 4. Conclusions.}

We now combine all these constructions and define $\mu_k$ in the whole of $\R^n$, namely
\begin{equation}
	\label{eq-measure-mu-k-ovunque}
	\mu_k:=\nu_k+\omega_k + \rho_k.
\end{equation}
We claim that $\mu_k$ satisfies the thesis. Indeed  \ref{zero item lemma on simplex} follows directly from the definition of $\nu_k$ in Step 1, while \ref{ii of lemma simplex} from the definition of $\omega_k$ and $\rho_k$ in Step 2 and Step 3 respectively. Property \ref{iii of lemma simplex} follows from  \eqref{eq-vanishing of segments1} and \eqref{eq-rho-tende a zero}.
As for \ref{first item lemma on simplex} it is a consequence of \eqref{convergence against bounded functions}, \ref{ii of lemma simplex} and \ref{iii of lemma simplex}.

To see  \ref{third item lemma on simplex} we simply write $
    \langle \mu_k,\nabla \varphi\rangle = \langle \nu_k+\omega_k,\nabla \varphi\rangle + \langle \rho_k,\nabla \varphi\rangle$ and recall
  \eqref{Lemma tetra-identità-massa sui bordi} and  \eqref{eq-div-rho-k}. 

\end{proof}

\begin{remark}
   We observe that each of the approximating measures $\mu_k$ constructed in the previous lemma are such that
    \begin{equation}\label{remark ineq mass approssimanti}
       \|\mu_k\|(K)\leq C \left [\mathcal L^3(T)+\frac1k\operatorname{diam}(T){\rm diam  }(K)+\frac1k(\operatorname{diam}(T))^3\right ]\sum_{j=1}^M|b_j|,
   \end{equation}
    for every compact set $K\subset \BR^3$.
\end{remark}

\begin{remark}\label{remark piecewise constant on normal vector}
Let $A(x)=\sum_{i=1}^M A_i\chi_{T^i}(x)$ be a piecewise constant function  with null divergence, where $A_i\in \BR^{N\times 3}$ and the $T^i$'s are tetrahedra. If $T^i$ and $T^j$ share a face, then, by integrating over a small cube across the common face, it is easy to see that it must hold  
\begin{equation}\label{equation matrix acting on normal vector}
    A_i \nu=A_j \nu,
\end{equation}
where $\nu$ is the unit normal vector of the common face.
\end{remark}

\begin{remark}\label{Remark explanation modification mass}
  A few comments on the measures constructed in Lemma \ref{lemma recovery on a simplex 2 } are
    in order. Recall the definition of $\Gamma^j_k$ and $\delta(k,h)$ given in Lemma \ref{lemma recovery on a simplex 2 }.
    With a direct computation one can see that, on average, the number of lines in $\Gamma^j_k$ intersecting $\delta(k,h)$ is $\hat N(k,j,h) :=\ha^2(\delta(k,h))k^4|\langle t^j,n_h\rangle|$. Consequently, given that $A=\sum_j b^j\otimes t^j$, the \emph{averaged mass} on each baricenter is 
    $$\hat B(T,k,h)=\sum_{j=1}^M \hat N(k,j,h) \operatorname{sign}(\langle t^j, n_h\rangle )\frac{b^j}{k^4} = \ha^2(\delta(k,h)) A n_h.$$  
    Furthermore, since \eqref{equation matrix acting on normal vector} holds, it is clear how the averaged mass is a more convenient boundary datum than the exact mass $B(T,k,h)$, as defined in \eqref{definition exact mass}. Indeed, in Lemma \ref{lemma approximation from piecewise constant to straight segment }, the averaged mass will allow us to glue together the local construction of Lemma \ref{lemma recovery on a simplex 2 } performed in different tetrahedra preserving the divergence free constraint.
    
    In this sense, the $\rho_k$'s in \eqref{eq-measure-mu-k-ovunque} are to be considered just a small correction necessary to pass from $B$ to $\hat B$.
    
\end{remark}

We now glue together the local construction of Lemma \ref{lemma recovery on a simplex 2 }  to obtain the global approximating sequence.

\begin{lemma}\label{lemma approximation from piecewise constant to straight segment }
    Let $\Omega \subset \BR^3$ be an open set, $\C T=\{T^1,\cdots,T^M\}$ be an admissible triangulaiton of $\Omega$ and $A=\sum_{i=1}^M A_i\chi_{T^i}$ be a divergence free piecewise constant function, with respect to $\C T$, where $A_i\in \BR^{N\times 3}$. Assume that $A_i=\sum_{j=1}^{ M^i} b_j^i\otimes t_j^i $, for some $b_j^i\in \C Q$, $ t_j^i\in \B S^2$,  then there exists a sequence of polyhedral measures $\mu_k\in \C M_{df}^1(\Omega;  \C Q \times \B S^2)$  such that 
    \begin{enumerate}
    	\item\label{lemma-39-1} $\mu_k \wsconv \operatorname{A}\C L^3 \mres \Omega$,
    	\item\label{lemma-39-2} there exists a sequence of  measures $\eta_k\in \C M^1(\Omega;  \C Q \times \B S^2)$ such that $\|\eta_k\|(K)\to 0$, for every compact set $K$, and 
  \begin{equation}
\label{eq-muk}    
\mu_k=\frac{1}{k^4}\sum_{i=1}^M  \sum_{j=1}^{M^i} b_j^i\otimes t_j^i  \ha^1\mres (\Gamma^{j,i}_k\cap  T^i_k) +\eta_k,
  \end{equation}
where $T^i_k\subset\subset T^i$ with ${\rm dist}(T^i_k,\partial T^i)\leq {\rm diam}(T^i)/k^2$, and $\Gamma_k^{j,i}$ is a union of straight lines parallel to $t^i_j$. Furthermore $\ha^1(\operatorname{supp}(\mu_k)\cap \partial T^i)=0$ for every $1\leq i\leq M$.
    \end{enumerate}
    
\end{lemma}
\begin{proof}
     For every tetrahedron $T^i\in \C T$ we apply Lemma \ref{lemma recovery on a simplex 2 } on $A_i$ to find four sequences of measures $\mu_k^i,\nu_k^i, \omega^i_k,\rho^i_k \in \C M^1_{{\rm loc}}(\R^3;\C Q\times \B S^2)$. Define $\mu_k:=\sum_{i=1}^M\mu_k^i$ and $\eta_k:=\sum_{i=1}^M \omega_k^i+\rho^i_k$.  Set $\Omega':={\rm int}(\cup_{T^i\in \C T} T^i)$. 
     
     We observe that from \ref{ii of lemma simplex} of Lemma \ref{lemma recovery on a simplex 2 }  , by taking an appropriate family of plane $\C F$ (containing all the boundaries $\partial T^i$), without loss of generality we can assume that $\ha^1(\operatorname{supp}\mu_k^i\cap\operatorname{supp}\mu_k^l)=0$ for $i\not=l$, and $\ha^1(\operatorname{supp}(\mu_k)\cap \partial T^i)=0$ for every $1\leq i\leq M$. In particular $\mu_k\in \C M^1(\Omega;  \C Q \times \B S^2)$ since each $\mu_k^i$ does and they have disjoint support.

    We first show that, with this definition, $\mu_k$ is divergence free in $\Omega'$. Indeed from \ref{third item lemma on simplex} of Lemma \ref{lemma recovery on a simplex 2 } we get that for every $\varphi \in [C^\infty_c(\Omega')]^N$
    \begin{equation*}
    	\langle \mu_k,\nabla \varphi\rangle= \sum_{i=1}^M  \langle \mu_k^i,\nabla \varphi\rangle=\sum_{i=1}^M\sum_{h=1}^{4k^2} \langle \ha^2(\Delta_i(k,h))A_i n^i_h ,\varphi(d(T_i,k,h))\rangle, 
    \end{equation*}
    where $n_h^i$ is the outer unit normal with respect to $T_i$, and  $\Delta_i(k,h)$ is one of the $4k^2$ triangles that tile $\partial T_i$ (as defined in \eqref{def-division in subsimplexes}). Then using that for any pair $i,j\in\{1,\ldots, M\}$ and $h,h'\in\{1,\ldots, 4k^2\}$ we have that either $\ha^2(\Delta_i(k,h)\cap \Delta_j(k,h'))=0$ or $\Delta_i(k,h)=\Delta_j(k,h^\prime)$ we denote the set
    $$
    \C A_k=\{(i,j,h,h'): \ i<j\,,\ \ha^2(\Delta_i(k,h)\cap \Delta_j(k,h'))>0\},
    $$
    and we can rewrite
    \begin{equation}
    	\langle \mu_k,\nabla \varphi\rangle= \sum_{\C A_k }\langle \ha^2(\Delta_i(k,h))(A_i n^i_h +A_{j} n_{h'}^{j}),\varphi(d(T_i,k,h))\rangle=0,
    \end{equation}
    since $n_h^i=-n_{h^\prime}^{j}$ and \eqref{equation matrix acting on normal vector} holds.

       Furthermore, from \ref{first item lemma on simplex} of Lemma \ref{lemma recovery on a simplex 2 } and the fact that $\lim_k\|\mu_k^i\|(\operatorname{int}T^i)=\lim_k\|\mu_k^i\|(\Omega)$,  we have that $\mu_k\wsconv A \C L^n $ in $\Omega$.
       
Finally property \ref{lemma-39-2} is a direct consequence of choice of $\C F$, the definition of $\mu_k^i $ and $\rho_k^i$, and \ref{zero item lemma on simplex} of Lemma \ref{lemma recovery on a simplex 2 }. 
\end{proof}

\begin{proof}[Proof of Theorem \ref{theo approximation straight lines}]
    Let $\mu\in [\C M (\Omega)]^{N\times 3}$ be divergence free. From Theorem \ref{prop Piecewise constant approximation} we obtain a sequence $\C T_k=\{T_k^i\}_{i=1,\cdots, M(k)}$ of admissible triangulations of $\Omega$ and a sequence $A_k:\Omega\to \mathbb R^{N\times 3}$ of piecewise constant functions relatively to $\C T_k$, such that $A_k \C L^3 \mres \Omega$ approximates strictly $\mu$. For each $k$, let $A_k(\Omega)=\{A_k^i\}_{i=1,\cdots, M(k)},$ and write $A_k^i=\sum (A_k^i)_{lm}e_l\otimes e_m$.  From Lemma \ref{lemma approximation from piecewise constant to straight segment } applied to each $A_k$  we then get a sequence of measures $\mu_k^h\in \C  M_{df}^1(\Omega;  \C Q \times \B S^2)$ approximating $A_k \C L^3\mres \Omega$ and such that, recalling \eqref{remark ineq mass approssimanti}, $\|\mu_k^h\|(\Omega)\leq C\|A_k\|_{L^1}\leq  C\|\mu\|(\Omega) $, hence we can find a diagonal sequence $\mu_k^{h(k)}\wsconv \mu$.
\end{proof}
\section{The upper bound}\label{sec-upper}

The approximation results proved above provide a local construction which is the crucial ingredient for the proof of the upper bound, which is stated and proved below.

\begin{proposition}
	\label{pro-upper-bound}
	Let $\psi:\B{Z}\times \B S^2 \longrightarrow [0,+\infty)$ be $\ha^1$-elliptic and obey $\frac{1}{c}|b|\leq \psi(b,t)$ for all $b\in \B Z$ and $t\in \B S^2$. Let $\Omega \subset \B R^3$ be a bounded open set, simply connected with  Lipschitz boundary. Then for every $\mu\in [\C M(\Omega)]^{N\times 3}$ with null divergence there exists a sequence $\mu_\sigma\in \C M_{df}^1(\Omega;\sigma\Z^N\times\B S^2)$ converging weakly$*$ to $\mu$ such that
	$$
	\liminf_{\sigma\to 0} E_{\sigma}(\mu_\sigma)\leq E_0(\mu).
	$$
\end{proposition}

\begin{proof}

The strategy of the proof follows closely the one in \cite{ConGarMul17}. It consists of a first step in which, using the approximation results, one reduces to divergence free measures concentrated on polyhedral curves whose limiting energy resolves the convexification procedure in the definition of the $\Gamma$-limit. With this we reduce the analysis to the construction of a recovery sequence for the auxiliary functional $F_\infty$ defined as follows:
\begin{equation}\label{eq-auxiliary-F}
	F_\infty(\mu, \omega):=\begin{cases}
		\displaystyle{\int_{\Gamma\cap \omega }\psi_\infty(\theta,t)d\ha ^1}\ \ &\hbox{if }\mu=\theta\otimes t \ha^1\mres \Gamma \in \C M_{df}^1(\omega;  \mathcal Q\times\mathbb S^2), 
		\\
		&\mu \hbox{ polyhedral}; \\
		+\infty  &\hbox{otherwise},
	\end{cases}
\end{equation}
where $\omega$ is an open set, and $F_\infty(\mu):=F_\infty(\mu,\Omega)$. 

\medskip

\textit{Step 1: Reduction from $E_0$  to $F_\infty$.}

We now prove that $E_0$ is the relaxation of $F_\infty$ with respect to the weak$*$ topology, i.e., $E_0=\bar F_\infty$.  Note that from the definition of $E_0$ we have that $E_0\leq F_\infty$, and hence $E_0\leq \bar F_\infty$, therefore we just have to prove the upper bound.
From Theorem \ref{prop Piecewise constant approximation} and Reshetnyak Continuity Theorem  we obtain that divergence free piecewise constant measures are dense in energy for $E_0$, i.e., for every divergence free measure $\mu\in [\C M(\Omega)]^{N\times 3}$ there exists a sequence of divergence free piecewise constant measures $\nu_k$ such that $\nu_k\wsconv
\mu$ and 
\begin{equation}
    \lim_{k\rightarrow+\infty}E_0(\nu_k)=E_0(\mu).
\end{equation}
Thus,  since the weak$*$ convergence is metrizable on bounded set of $\mes$, without loss of generality we can now assume $\mu$ to be a divergence free piecewise constant measure of the form
\begin{equation}\label{upper bound structtura piecewise constant}
    \mu=\sum_{i=1}^M \chi_{T^i}A_i \C L^3,
\end{equation}
where $\mathcal T=\{T^1,\cdots,T^M\}$ is an admissible triangulation of $\Omega$.
We thus construct the recovery sequence in the relaxation of $F_\infty$ for a measure $\mu$ as in \eqref{upper bound structtura piecewise constant}.

First recall that $g$ is the convex envelope of $g_\infty$. Moreover we know (see Lemma \ref{lemma properties g}) that $g$ is finite and that  $\psi_\infty(b,t)=+\infty$ if and only if $b\in \BR^N\smallsetminus \C{Q}$. Therefore  for any  fixed $\eps>0$ and for every matrix $A_i$, with $i\in\{1,\ldots, M\}$, we find  $M^i_\eps\leq 3N+1$ rank one matrices of the form $\tilde b^{j,i}_\eps\otimes t^{j,i}_\eps $, with  $\tilde b^{j,i}_\eps\in \C Q$, $t^{j,i}_\eps\in \mathbb S^2$, and coefficients $\lambda^{j,i}_\eps>0$ such that $\sum_{j=1}^{M^i_\eps}\lambda^{j,i}_\eps=1$, $\sum_{j=1}^{M^i_\eps}\lambda^{j,i}_\eps \tilde b^{j,i}_\eps\otimes t^j_\eps=A_i$, and
\begin{equation}\label{estimate for g in the approximation in the integral}
	g(A_i)+\eps\geq \sum_{j=1}^{M^i_\eps}\lambda^{j,i}_\eps\psi_\infty (\tilde b^{j,i}_\eps,t^{j,i}_\eps)= \sum_{j=1}^{M^i_\eps}\psi_\infty (\lambda^{j,i}_\eps\tilde b^{j,i}_\eps,t^{j,i}_\eps),
\end{equation}
Setting $b^j_\varepsilon:=\lambda^j_\varepsilon\Tilde{b_\varepsilon}^j \in \C Q$ we then have 
\begin{equation}\label{estimate for g in the approximation in the integral due}
	 (g(A_i)+\eps)\C{L}^3(T^i\cap \Omega)\geq \sum_{j=1}^{M^i_\eps}\psi_\infty (b^{j,i}_\eps,t^{j,i}_\eps)\C{L}^3(T^i\cap \Omega).
\end{equation}

We now apply Lemma \ref{lemma approximation from piecewise constant to straight segment } 
to $\mu=\sum_{i=1}^M \chi_{T^i}A_i$, with $A_i=\sum_{j=1}^{M_\eps^i}b^{j,i}_\eps\otimes t^{j,i}_\eps$ satisfying \eqref{estimate for g in the approximation in the integral}, to find a sequence of  measures $\mu_k^\eps\in \C M_{df}^1(\Omega ; \C Q\times \B S^2)$ converging to $\mu$, and $\eta_k^\eps$ vanishing as  $k\to +\infty$.

On the other hand, from \ref{ii of lemma simplex} of Lemma \ref{lemma recovery on a simplex 2 } and using that $\psi_\infty(\theta,t)\leq C|\theta|$ for all $\theta \in \C Q$ it is easy to see that $F_\infty(\mu_k^\eps,\operatorname{int}(T^i)\cap \Omega)< +\infty$ and
\begin{equation}
    F_\infty(\mu_k^\eps,\operatorname{int}(T^i)\cap \Omega)\leq\sum_{j=1}^{M^i_\eps}\psi_\infty(b^{j,i}_\eps,t^{j,i}_\eps)\frac{1}{k^4}\ha^1(T_k^i\cap \Omega \cap \Gamma^{j,i,\eps}_{k}) + C\|\eta_k^\eps\|(T^i\cap \Omega).
\end{equation}
 Hence from \eqref{lemma-item-1-weak-convergence} of Lemma \ref{lemma recovery on a simplex 2 } we get
\begin{equation}
\limsup_{k\rightarrow +\infty}F_\infty(\mu_k^\eps,\operatorname{int}(T^i)\cap \Omega)\leq \sum_{j=1}^{M^i_\eps}\psi_\infty (b^{j,i}_\eps,t^{j,i}_\eps)\C{L}^3(T^i\cap \Omega).
\end{equation}
Recalling \eqref{estimate for g in the approximation in the integral due} we then get
\begin{align*}
E_0(\mu)+\eps\mathcal L^3(\Omega)&=\sum_{i=1}^M(g(A_i)+\eps)\C{L}^3(T^i\cap \Omega)\geq \sum_{i=1}^M\limsup_{k\rightarrow\infty}F_\infty(\mu_k^\eps,\operatorname{int}(T^i)\cap \Omega)\\
     &\geq \limsup_{k\rightarrow\infty}\sum_{i=1}^MF_\infty(\mu_k^\eps,\operatorname{int}(T^i)\cap \Omega) =  \limsup_{k\rightarrow\infty}F_\infty(\mu^\eps_k,\Omega),
\end{align*}
where in the last line we used the fact that $\ha^1(\operatorname{supp}(\mu_k)\cap \partial T^i)=0 $ for all $i$.

Now
 \ref{lemma convergence psi to psi infty} of Proposition \ref{prop-recession-properties} implies  $\psi_\infty(b,t)\geq c |b|$ for $b\in \C Q$, hence, from estimate \eqref{estimate for g in the approximation in the integral} we deduce the existence of a universal constant $C>0$ such that 
\begin{equation}\label{inequality boundedness b_j}
    \sum_{j=1}^{N_\eps}|b_\eps^{j,i}| \leq  C(|A_i|+\eps).
\end{equation}
In particular, given \eqref{remark ineq mass approssimanti}, we have that the family of measures $\mu^\eps_{k}$ is uniformly bounded.
Therefore, since the weak$*$ topology is metrizable on bounded set, via a diagonal argument we infer that for every divergence free measure  $\mu\in [\C M(\Omega)]^{N\times 3}$ there exists a sequence $\tilde \mu_h$ weakly$*$ converging to $\mu$, such that 
\begin{equation}\label{eq-upper-1}
  \limsup_{h\to +\infty} F_\infty(\tilde\mu_h)\leq E_0(\mu).
\end{equation}
In particular this sequence satisfies $F_\infty(\tilde \mu_h)<+\infty$.
\medskip

\textit{Step 2: Recovery sequence for $F_\infty$.}

We  now prove that for a given $\tilde\mu= \theta\otimes t \ha^1\mres \Gamma\in \C M_{df}^1(\Omega; \R^N\times\mathbb S^2)$, with $\Gamma$ polyhedral  we can construct  a sequence $\hat\mu_\sigma \in \C M_{df}^1(\Omega; \sigma\Z^N\times\mathbb S^2)$ converging to $\tilde \mu$ and such that 
\begin{equation}
	\label{eq-intermedio-upper}
	\limsup_{\sigma \rightarrow 0} E_\sigma(\hat\mu_\sigma)= F_\infty(\tilde\mu).
\end{equation}
Without loss of generality we can assume $F_\infty(\tilde \mu)<+\infty$. First we observe that, since $\Gamma$ is composed by a finite number of straight segments and $\tilde\mu$ is divergence free, $\theta$ must be constant on each segment. In particular $\theta\in \C Q $ attains a finite number of values. Therefore we can apply Theorem $2.5$ of \cite{ConGarMas15} to $\mu$, deducing that there exists a finite number of polyhedral closed loops $\Gamma_i$ with constant Burgers vector $\theta_i$ such that 
$$
\tilde\mu=\sum_i \theta_i\otimes t_i\ha^1\mres \Gamma_i.
$$
Here the $\theta_i$ do not necessary belong to $\sigma\Z^N$. Therefore we define the following  approximation of $\tilde\mu$ with measures 
\begin{equation}
	\hat\mu_\sigma:=\sum_i\sigma \left\lfloor \dfrac{\theta_i}{\sigma}\right\rfloor\otimes t_i  \ha^1\mres \Gamma_i \in \C M_{df}^1(\Omega;\sigma \Z^N\times \B S^2),
\end{equation}
where we denote $\left\lfloor \frac{b}{\sigma}\right\rfloor =(\big\lfloor \frac{b_1}{\sigma}\big\rfloor,\cdots,\big\lfloor \frac{b_N}{\sigma}\big\rfloor)$. 
These measures have the same support of $\mu$, and satisfy $E_\sigma(\hat\mu_\sigma)< +\infty$. We observe that $\hat\mu_\sigma$ is a finite sum of closed loops with constant multiplicity, therefore, again by Theorem $2.5$ in \cite{ConGarMas15}, it is divergence free. 

Since $\sigma \left\lfloor \frac{\theta_i}{\sigma}\right\rfloor $ converges to $ \theta_i$, we have that $\hat\mu_\sigma \wsconv \tilde\mu$.
Let now consider an arbitrary sequence $\sigma_j$ converging to $0$, then

\begin{equation}
	E_{\sigma_j}(\hat\mu_{\sigma_j})=\int_{\Gamma} \sigma_j\psi  ( z_j ,t ) d \ha^1,
\end{equation}
 where $z_j(x):= \sum_i \left\lfloor \frac{\theta_i}{\sigma_j}\right\rfloor \chi_{\Gamma_i}(x) \in \Z^N $.
 Since, clearly, for $\ha^1\mres \Gamma $ a.e. $x$ it holds that $\lim_{j\to+\infty} |z_j|=+\infty$ and $\lim_{j\to+\infty}  z_j/|z_j|=\theta/|\theta|$, thanks to  \ref{lemma convergence psi to psi infty} of Lemma \ref{prop-recession-properties} we deduce  
 $$
 \lim_{j\to+\infty}  \frac{\psi\left (z_j ,t\right )}{|z_j|}=\frac{\psi_\infty\left (\theta ,t\right )}{|\theta|}.
 $$ 
 On the other hand clearly it holds that $\lim_{j\to+\infty} \sigma_j|z_j|=|\theta|$, hence by rewriting 
\begin{equation} \label{eq-rewrite-psi}
	\sigma_j\psi\left (z_j,t\right )=\sigma_j|z_j|\frac{\psi\left (z_j ,t\right )}{|z_j|},
\end{equation}
we deduce that 
\begin{equation}
	\lim_{j\rightarrow +\infty} \sigma_j\psi( z_j,t)= \psi_\infty\left (\theta,t\right ),\ \ \ \ \ha^1\textit{-a.e. } x\in \Gamma.
\end{equation}

Finally, since  $|\sigma_j\psi( z_j ,t_i )|\leq c|\sigma_j z_j|\leq C (|\theta|+1)$, we conclude via dominated convergence theorem and get
\begin{equation}
	\lim_{j\to +\infty} E_{\sigma_j}(\hat\mu_{\sigma_j})=\int_\Gamma\psi_\infty(\theta,t)d \ha^1=F_\infty(\tilde\mu).
\end{equation}
\medskip

\textit{Step 3: Conclusion}

We conclude the proof by combining Step 1 and Step 2. Without loss of generality we can assume that the sequence $\tilde\mu_h$ constructed in Step 1 satisfies
\begin{equation}\label{eq-conclusion1}
\limsup_{h\to +\infty} F_\infty(\tilde\mu_h)=\lim_{h\to +\infty} F_\infty(\tilde\mu_h)\leq E_0(\mu)<+\infty.
\end{equation}
Then for every $h$ we obtain via Step 2 a sequence $\hat\mu_{\sigma}^h$ weakly$*$ converging to $\tilde\mu_h$ such that
\begin{equation}\label{eq-conclusion2}
\lim_{\sigma\to 0} E_0(\hat\mu_\sigma^h)= F_\infty(\tilde\mu_h).	
\end{equation}
A further diagonal argument provides the wanted recovery sequence and concludes the proof.
\end{proof}

\section*{Fundings}
The present paper benefits from the support of the PRIN 2017, Variational methods for stationary and evolution problems with singularities and interfaces,  2017BTM7SN-004. Also MF acknowledges the hospitality given by the HIM,  Hausdorff Research Institute for Mathematics.\\
\ \ \\
Declarations of interest: none.
\bibliographystyle{plain}
\bibliography{ForGar23}
\end{document}